\documentclass[a4paper,reqno,12pt]{amsart}
\usepackage[dvips]{graphicx}
\usepackage{graphics}
\usepackage{amsthm, amssymb, amsmath, amsfonts}

\newtheorem{thm}{Theorem}[section]

\newtheorem{prop}{Proposition}[section]
\newtheorem{cor}{Corollary}[section]

\newtheorem{rmk}{Remark}[section]

\newcommand{\R}{{\mathbb R}}
	
	\newcommand{\C}{{\mathbb C}}
	\newcommand{\ds}{\displaystyle}
	\newcommand{\To}{\longrightarrow}
	\DeclareMathOperator*\re{Re}
	\DeclareMathOperator*\im{Im}
\newcommand{\p}{\partial}
\def\1{\^{\i}}
\def\2{\u{a}}
\def\3{\c{s}}
\def\4{\^{a}}
\def\5{\c{t}}

\def\b{\beta}

\def\g{\gamma}

\def\<{\langle}
\def\>{\rangle}

\begin{document}
\title{Injectivity and univalence of complex functions via  monotonicity}
\author{ Szil\'ard L\' aszl\' o$^*$}
\thanks{$^*$This work was supported by a grant of the Romanian Ministry of Education, CNCS - UEFISCDI, project number PN-II-RU-PD-2012-3 -0166.}
\address{S. L\' aszl\' o, Department of Mathematics, Technical University of Cluj Napoca,
Str. Memorandumului nr. 28, 400114 Cluj-Napoca, Romania.}
\email{szilard.laszlo@math.utcluj.ro}

\begin{abstract}
In this paper we provide sufficient conditions that ensure the monotonicity, respectively the global injectivity of an operator. Further, some new analytical conditions that assure the injectivity/univalence of a complex function of one complex variable are obtained. We also show that some classical results, such as Alexander-Noshiro-Warschawski and Wolff theorem or Mocanu theorem, are easy consequences of our results.
\newline


\end{abstract}

\maketitle

\section{Introduction}

One of the most celebrated  results that provides the univalence of a holomorphic function $f:D\subseteq \C\To \C,\, f=u+iv,$ is $\re f'(z)>0$ for all $z\in D.$ However, it can  easily be shown (see for instance \cite{KaPiLa1}) that this condition is equivalent to the strict monotonicity of the vector function $f=(u,v),$ and it is well known that  strictly monotone operators are injective. On the other hand the mentioned univalence condition is a particular case of Alexander-Noshiro-Warschawski and Wolff theorem (for $\g=0$, see \cite{Mb,No,Wa,Wo}), and the latter cannot be deduced by using the  classical strict monotonicity concept of an operator.

Let us mention that several injectivity conditions for operators that are monotone in some sense were obtained recently  in \cite{KaPiLa1,KaPiLa2} and \cite{Pi}. These results were applied then to obtain some injectivity/univalence results for complex functions of one complex variable. In this paper we deal with operators which are monotone relative to another operator. We obtain some sufficient (analytical) conditions that ensure this monotonicity property. We also show that operators having this monotonicity property are injective under some circumstances. By combining the mentioned results we obtain some analytical conditions that ensure injectivity of an operator.
We also extend the well-known injectivity result expressed in terms of positive definiteness of the symmetric part of all Fr\`echet differentials of operators of class $C^1$.

In the last section we apply these results to obtain some unknown injectivity, respectively univalence results for complex functions of one complex variable. As immediate consequences of our main result we obtain Mocanu theorem, respectively Alexander-Noshiro-Warschawski and Wolff theorem concerning on injectivity, respectively univalence of complex functions.

\section{Analytical conditions for monotonicity and injectivity}

 Let $\big(H,\<\cdot,\cdot\>\big)$ be a real Hilbert space identified with its  topological dual.  Consider the operator 	$T:D\subseteq H\To H$ and let $A:H\To H$ be another operator. We say that the operator $T$ is {\it monotone  relative to}  $A$ if for all $x,y\in D$ one has
	\begin{equation}\label{e1}
	\<T(x)-T(y),A(x)-A(y)\>\ge 0.
	\end{equation}
	$T$ is called {\it strictly monotone relative to} $A$ if in (\ref{e1}) equality holds only for $A(x)=A(y).$

For $x,y\in H$ let us denote by $(x,y)$ the open line segment with the endpoints $x,$ respectively $y$, i.e.
$$(x,y)=\{x+t(y-x):0<t<1\}.$$
Let $D\subseteq H$ be open. For a differentiable operator $T:D\To H$, we denote by $dT_x(\cdot)$ the Fr\`{e}chet differential of $T$ at $x\in D.$ In what follows we provide an analytical condition that ensures the monotonicity of an operator relative to another operator.
	
\begin{prop}\label{p31} Let $D\subseteq H$ be an open and convex set, let  $T:D\To H$ be an operator of class $C^1$ and let $A:H\To H$ be an operator. Assume that for all $x,y\in D,\mbox{ with } A(x)\neq A(y)$ and $z\in (x,y)$ one has
$$\<dT_{z}(y-x),A(y)-A(x)\>>0.$$
Then $T$ is strictly monotone relative to $A$.
\end{prop}
\begin{proof}  Let $x,y\in D$ such that $A(x)\neq A(y).$ We show that $\<T(y)-T(x),A(y)-A(x)\>>0.$

Consider the real function $\phi:[0,1]\To\R,\, \phi(t)=\<T(x+t(y-x)),A(y)-A(x)\>.$ Then $\phi$ is contionuous on $[0,1]$ and differentiable on $(0,1)$, hence according to mean value theorem, there exists $c\in(0,1)$ such that $\phi'(c)=\phi(1)-\phi(0).$ Equivalently, we can state
$$\<dT_{x+c(y-x)}(y-x),A(y)-A(x)\>=\<T(y)-T(x),A(y)-A(x)\>.$$
On the other hand $c\in(0,1)$ implies $x+c(y-x)\in(x,y)$ and by the hypothesis of theorem we have
$$\<dT_{x+c(y-x)}(y-x),A(y)-A(x)\>>0,$$
which shows that
$$\<T(y)-T(x),A(y)-A(x)\>>0.$$
\end{proof}

\begin{rmk}\label{r31}\rm Note that  the condition $\<dT_{z}(y-x),A(y)-A(x)\><0$  for all $x,y\in D,\,x\neq y$ and $z\in (x,y)$  ensures that $T$ is strictly monotone relative to $-A.$

It can be  analogously proved that the condition $\<dT_{z}(y-x),A(y)-A(x)\>\ge0$ for all $x,y\in D$ and $z\in (x,y),$ ensures that $T$ is  monotone  relative to $A$.
\end{rmk}

For a linear operator $A:H\To H$ we denote by $\ker A$ the set of zeroes of $A$, that is
$$\ker A=\{x\in H: A(x)=0\}.$$
\begin{cor}\label{c31} Let $D\subseteq H$ be an open and convex set, let  $T:D\To H$ be an operator of class $C^1$ and let $A:H\To H$ be a linear operator. Assume that for all $x\in D$ and $y\in H\setminus\ker A$ one has
$$\<dT_{x}(y),A(y)\>>0.$$
Then $T$ is strictly monotone relative to $A$.
\end{cor}
\begin{proof} Indeed, let $ u,v\in D,\mbox{ with } A(u)\neq A(v).$ Take $y=v-u$  and $x=w\in(u,v).$ Since  $A(u)\neq A(v)$ we have $A(y)\neq 0$, hence $y\in D\setminus\ker A.$

Consequently, the condition $\<dT_{x}(y),A(y)\>>0$  for all $y\in H\setminus\ker A$ becomes
$$\<dT_{w}(v-u),A(v)-A(u)\>>0,\,\forall u,v\in D,\mbox{ with } A(u)\neq A(v).$$
The conclusion follows from Proposition \ref{p31}.
\end{proof}

\begin{rmk}\label{r32}\rm If we assume that for all $x\in D$ and $y\in H\setminus \ker A$ one has
$\<dT_{x}(y),A(y)\><0,$
we obtain that $T$ is strictly monotone relative to $-A$.
\end{rmk}

Next we provide some conditions that ensure the injectivity of an  operator which is monotone relative to an operator $A$.

\begin{prop}\label{p32} An operator $T:D\subseteq H\To H$  which is strictly monotone relative to $A:H\To H,$  is injective on $D\setminus\{x\in D: \exists y\in D,\,x\neq y,\,A(x)=A(y)\}.$  If $A$ is injective on $D$, that is, for all $x,y\in D,\, x\neq y$ one has $A(x)\neq A(y)$, then $T$  is also injective on its whole domain.
\end{prop}
\begin{proof} Indeed, for $u,v\in D\setminus\{x\in D: \exists y\in D,\,x\neq y,\,A(x)=A(y)\},\,u\neq v$ one has $A(u)\neq A(v)$, hence
$$\<T(u)-T(v),A(u)-A(v)\>>0.$$
The latter relation shows that $T(u)\neq T(v).$

If $A$ is injective on $D$, then $\{x\in D: \exists y\in D,\,x\neq y,\,A(x)=A(y)\}=\emptyset ,$ hence $T$ is injective on $D.$
\end{proof}

Combining the results obtained so far we obtain the following.

\begin{thm}\label{t31} Let $D\subseteq H$ be an open and convex set, let  $T:D\To H$ be an operator of class $C^1$ and let $A:H\To H$ be an   operator injective on $D$. Assume that one of the following conditions hold.
\begin{itemize}
\item[(a)] For all $x,y\in D$ with $A(x)\neq A(y)$ and $z\in (x,y)$ one has
$$\<dT_{z}(y-x),A(y)-A(x)\>>0.$$
\item[(b)] $A$ is linear and for all $x\in D$ and $y\in H\setminus\ker A$ one has
$$\<dT_{x}(y),A(y)\>>0.$$
\end{itemize}
Then $T$ is injective.
\end{thm}
\begin{proof} The conclusion follows from Proposition \ref{p31} and Proposition \ref{p32}, respectively  from Corollary \ref{c31} and Proposition \ref{p32}.
\end{proof}

\begin{rmk}\label{r33}\rm Since $A$ is injective on $D$ if and only if $-A$ is injective on $D$, according to Remark \ref{r31}, respectively Remark \ref{r32}, the conditions

\item[(a)] For all $x,y\in D,\, A(x)\neq A(y)$ and $z\in (x,y)$ one has
$$\<dT_{z}(y-x),A(y)-A(x)\><0,$$
 respectively
 \item[(b)] $A$ is linear and for all $x\in D$ and $y\in H\setminus \ker A$ one has
$$\<dT_{x}(y),A(y)\><0,$$
also assure the injectivity of $T.$
\end{rmk}

Consider now $H=\R^n$ endowed with the usual euclidian scalar product, let $T:D\subseteq \R^n\To\R^n,\,T=(t_1,t_2,\ldots,t_n)$
 be an operator of class $C^1$ and let  $A:\R^n\To\R^n$ be a linear operator. For $x^0=(x_1^0,x_2^0,\ldots x_n^0)\in D$ we denote by $J_T(x^0)$ the Jacobian matrix of $T$ in $x^0$, i.e.
$$J_T(x^0)=\ds\begin{pmatrix}\ds \frac{\p t_1}{\p x_1}(x^0)&\ds\frac{\p t_1}{\p x_2}(x^0)&\cdots&\ds \frac{\p t_1}{\p x_n}(x^0)\\
 \ds\frac{\p t_2}{\p x_1}(x^0)&\ds\frac{\p t_2}{\p x_2}(x^0)&\cdots& \ds\frac{\p t_2}{\p x_n}(x^0)\\
\vdots&\vdots&\cdots&\vdots\\
\ds\frac{\p t_n}{\p x_1}(x^0)&\ds\frac{\p t_n}{\p x_2}(x^0)&\cdots&\ds \frac{\p t_n}{\p x_n}(x^0)
\end{pmatrix}.$$

Note that the linear operator $A$ can be identified with a real square matrix $(a_{ij})_{1\le i,j\le n}$. Let us denote by $A^\top$ the transpose of $A.$ For a given square matrix $B$ of order $n$ we denote  the submatrix obtained by deleting the last $n-m$ rows and the last $n-m$ columns by $(B)_{1\le i,j\le m}.$

\begin{thm}\label{t32} Let $D\subseteq \R^n$ be an open and convex set, let  $T:D\To \R^n$ be an operator of class $C^1$ and let $A:R^n\To \R^n$ be a linear  operator with $\det A\neq 0.$ Assume that for all $x\in D$  one of the following conditions hold.
\begin{itemize}
\item[(a)] $\det(A^\top J_T(x)+ J^\top_T(x)A)_{1\le i,j\le m}>0,\, \forall m\in\{1,2,\ldots,n\}.$
\item[(b)] $(-1)^m\det(A^\top J_T(x)+ J^\top_T(x)A)_{1\le i,j\le m}>0,\, \forall m\in\{1,2,\ldots,n\}.$
\end{itemize}
Then $T$ is injective.
\end{thm}
\begin{proof}  Let $x\in D.$ Then we have $\<dT_{x}(y),A(y)\>=\<A^\top dT_{x}(y),y\>=y^\top A^\top J_T(x) y.$
This shows that the positive definiteness, respectively negative definiteness of $A^\top J_T(x)$ is equivalent to
$$\<dT_{x}(y),A(y)\>>0,\,\forall y\in \R^n\setminus\{0\},$$
respectively
$$\<dT_{x}(y),A(y)\><0,\,\forall y\in \R^n\setminus\{0\}.$$

On the other hand, $y^\top A^\top J_T(x) y=((A^\top J_T(x))^\top y)^\top y=(J_T^\top(x)Ay)^\top y=y^\top ((J_T^\top(x)Ay)^\top)^\top=y^\top J_T^\top(x)Ay,$ hence
$$y^\top A^\top J_T(x) y=\frac12 y^\top \left(A^\top J_T(x)+ J_T^\top(x)A\right)y.$$
 %Note that   the condition $\<dT_{x}(y),A(y)\>>0$ for all $y\in \R^n$ is equivalent to the positive definiteness of $A^\top J_T(x).$

 Observe that $$\det(A^\top J_T(x)+ J^\top_T(x)A)_{1\le i,j\le m}>0,\, \forall m\in\{1,2,\ldots,n\}.$$
is actually  Sylvester's criterion for positive definiteness  of the symmetric matrix
$$A^\top J_T(x)+ J^\top_T(x)A,$$
meanwhile the condition $(-1)^m\det(A^\top J_T(x)+ J^\top_T(x)A)_{1\le i,j\le m}>0,\, \forall m\in\{1,2,\ldots,n\}$ is actually  Sylvester's criterion for negative definiteness  of the symmetric matrix
$$A^\top J_T(x)+ J^\top_T(x)A.$$

But the latter relations are equivalent to the positive definiteness, respectively negative definiteness  of  $A^\top J_T(x).$
Hence, we have
$$\<dT_{x}(y),A(y)\>>0,\,\forall y\in \R^n\setminus\{0\},$$
respectively
$$\<dT_{x}(y),A(y)\><0,\,\forall y\in \R^n\setminus\{0\}.$$
Since $\det A\neq 0$ we obtain that $A$ is injective. The conclusion follows from Theorem \ref{t31}, respectively Remark \ref{r33}.
\end{proof}

\section{Injective complex functions}
	
Let us denote by $\C$ the set of complex numbers, that is
$$\C=\{z=x+iy:x,y\in\R,\, i^2=-1\}.$$
For $z=x+iy\in \C$ we denote by $\re z,\, \im z,\,\overline{z}$, respectively $|z|$ the real part, imaginary part, conjugate and absolute value respectively, that is $\re z=x,\,\im z=y,\,\overline{z}=x-iy$ and $|z|=\sqrt{x^2+y^2}.$ Obviously $z\overline{z}=|z|^2.$   Note that the real linear space $\C$ becomes a real Hilbert space with the inner product
$$\<z,w\>=\re z\overline{w}.$$
This real Hilbert space may be identified with the real Hilbert space $\R^2$ endowed with the euclidean scalar product, therefore we can identify $z\in \C$ by $(\re z,\im z)\in\R^2.$

	 Let $D\subseteq \C$ be open. For a complex function of one complex variable $f:D\To\C,\, f(z)=u(x,y)+iv(x,y),\,\forall z=x+iy\in D$  of class $C^1(D),$ we denote by $J_f(z_0)$ the Jacobian matrix of $f$ in $z_0=x_0+iy_0$, i.e. $$J_f(z_0)=\begin{pmatrix} u_x\rq{(x_0,y_0)}&u_y\rq{(x_0,y_0)}\\v_x\rq{(x_0,y_0)}&v_y\rq{(x_0,y_0)}\end{pmatrix}.$$
	 If we consider $f$ as the vector function $(u,v)$ then its differential in $z_0=x_0+iy_0$ can be defined as
	 $$df_{(x_0,y_0)}(p,q)=J_f(z_0)\cdot\begin{pmatrix} p\\q\end{pmatrix},$$
	 hence for $w=p+iq$ the differential of
	$f$ in $z_0$ becomes $$df_{z_0}(w)=(u_x\rq{(x_0,y_0)}p+u_y\rq{(x_0,y_0)}q)+i(v_x\rq{(x_0,y_0)}p+v_y\rq{(x_0,y_0)}q).$$
	The partial derivatives of $f$ are defined as:
	$$\frac{\p f}{\p x}(z_0)=u_x\rq{(x_0,y_0)}+iv_x\rq{(x_0,y_0)},$$
 respectively
 $$\frac{\p f}{\p y}(z_0)=u_y\rq{(x_0,y_0)}+iv_y\rq{(x_0,y_0)}.$$
	Let us introduce the following notations:
	$$\frac{\p f}{\p z}(z_0)=\frac12\left(\frac{\p f}{\p x}(z_0)-i\frac{\p f}{\p y}(z_0)\right),$$
 respectively
 $$\frac{\p f}{\p \overline{z}}(z_0)=\frac12\left(\frac{\p f}{\p x}(z_0)+i\frac{\p f}{\p y}(z_0)\right).$$

The main result of this section is the following general injectivity result.

\begin{thm}\label{t41} Let $D\subseteq\C$ be open and convex and let $f:D\To\C$ be a function of class $C^1.$ Assume that there exist $w_1,w_2\in\C$ such that $\re w_1\im w_2\neq \re w_2\im w_1$ and for all $z\in D$ the following condition holds:
\begin{equation}\label{e3}
\re\left(\frac{\p f}{\p z}(z)w_1+\frac{\p f}{\p \overline{z}}(z)\overline{w_1}\right)+\im\left(\frac{\p f}{\p z}(z)w_2+\frac{\p f}{\p \overline{z}}(z)\overline{w_2}\right)>
\end{equation}
$$ \left|\frac{\p f}{\p z}(z)(w_2-iw_1)+\frac{\p f}{\p \overline{z}}(z)\overline{w_2+iw_1}\right|.
$$  Then $f$ is injective.
\end{thm}
\begin{proof} One can assume that $w_1=a+ib,\,w_2=c+id,\,a,b,c,d\in\R.$ It can easily be deduced, that $(\ref{e3})$ is equivalent to
$$(u'_xa+u'_yb)+(v'_xc+v'_yd)>$$
$$\sqrt{\left((u'_xc+u'_yd)+(v'_xa+v'_yb)\right)^2+\left((v'_xc+v'_yd)-(u'_xa+u'_yb)\right)^2}.$$
By taking the square of both sides we obtain
$$4(u'_xa+u'_yb)(v'_xc+v'_yd)>\left((u'_xc+u'_yd)+(v'_xa+v'_yb)\right)^2,$$
or equivalently
$$4\re df_z(w_1) \cdot \im df_z(w_2)>(\re df_z(w_2) + \im df_z(w_1))^2,\,\forall z\in D.$$
The latter relation can be written as
\begin{equation}\label{e4}
\det\begin{pmatrix} 2\re df_z(w_1)&\re df_z(w_2)+\im df_z(w_1)\\\re df_z(w_2)+\im df_z(w_1)&2\im df_z(w_2)\end{pmatrix}>0,\,\forall z\in D.
\end{equation}

 Let us denote by $L$ the matrix $\begin{pmatrix} a&c\\b&d\end{pmatrix}.$ Then $(\ref{e4})$ becomes
\begin{equation*}
\det(L^\top J_f(z)+J^\top_f(z)L)>0,\, \forall z\in D .
\end{equation*}
%Note that $\re df_z(w_1)>0,\,\forall z\in D$ and $(\ref{e5})$ assure the positive definiteness of the symmetric matrix $L^\top J_f(z)+J^\top_f(z)L$ for all $z\in D$, meanwhile $\re df_z(w_1)<0,\,\forall z\in D$ and $(\ref{e5})$ assure the negative definiteness of  the symmetric matrix $L^\top J_f(z)+J^\top_f(z)L$ for all $z\in D$. %On the other hand it is well known that the positive (negative) definiteness of a square matrix $M$ is equivalent to the positive (negative) definiteness of the symmetric matrix $A+A^t.$

%Hence, $\re df_z(w_1)>0,\,\forall z\in D$ and $(\ref{e5})$ assure the positive definiteness of $L^\top J_f(z)+J^\top_f(z)L$ for all $z\in D$ , meanwhile $\re df_z(w_1)<0,\,\forall z\in D$ and $(\ref{e5})$ assure the negative definiteness of $L^\top J_f(z)+J^\top_f(z)L$ for all $z\in D$.

We show next that $\re df_z(w_1)>0$ for all  $z\in D,$ or $\re df_z(w_1)<0$ for all  $z\in D.$ Observe that $(\ref{e4})$ assures that $\re df_z(w_1)\neq 0$ for all $z\in D.$ Assume for instance that $\re df_{z_1}(w_1)>0$ and $\re df_{z_2}(w_1)<0$ for some $z_1,z_2\in D.$ Then, the intermediate value theorem, applied to the function $g:D\To\R,\, g(z)=\re df_z(w_1),$ provides the existence of $z_3\in D$ such that $\re df_{z_3}(w_1)=0,$ contradiction.

In conclusion one of the following conditions is fulfilled.
\begin{itemize}
 \item[(i)] $\re df_z(w_1)>0$ and  $\det(L^\top J_f(z)+J^\top_f(z)L)>0$  for all  $z\in D,$ or
 \item[(ii)] $\re df_z(w_1)<0$ and  $\det(L^\top J_f(z)+J^\top_f(z)L)>0$  for all  $z\in D.$
 \end{itemize}
 Note that $(i)$, respectively $(ii)$ are equivalent to
 \begin{itemize}
 \item[(a)] $\det(L^\top J_f(z)+J^\top_f(z)L)_{1\le i,j\le m}>0,\, \forall m\in\{1,2\},$
respectively
\item[(b)] $(-1)^m\det(L^\top J_f(z)+J^\top_f(z)L)_{1\le i,j\le m}>0,\, \forall m\in\{1,2\}.$
\end{itemize}

Since $\re w_1\im w_2\neq \re w_2\im w_1,$ we obtain that $L$ is invertible,  hence injective. According to Theorem \ref{t32}  $f$ is injective.
\end{proof}

\begin{rmk}\label{r4}\rm
One can easily deduce that for $z\in D$ and $w\in\C$ we have
	$$df_{z}(w)=\frac{\p f}{\p z}(z)w+\frac{\p f}{\p \overline{z}}(z)\overline{w},$$
hence the condition $(\ref{e3})$ in the hypothesis of Theorem \ref{t41}  can be replaced by
$$\re df_z(w_1)+\im df_z(w_2)>|df_z(w_2)-idf_z(w_1)|,\,\forall z\in D.$$
\end{rmk}

The next Corollary can be viewed as an extension of Mocanu's injectivity result.

\begin{cor}\label{c41} Let $D\subseteq\C$ be open and convex and let $f:D\To\C$ be a function of class $C^1.$ Assume that there exist $\g\in\R$ such that, for all $z\in D$ it holds:
$$
\re\left(\frac{\p f}{\p z}(z)e^{i\g}\right)>\left|\frac{\p f}{\p \overline{z}}(z)\right|.$$
  Then $f$ is injective.
\end{cor}
\begin{proof} Take $w_1=e^{i\g}$ and $w_2=ie^{i\g}.$ Then an easy computation shows, that
$$
2\re\left(\frac{\p f}{\p z}(z)e^{i\g}\right)=\re\left(\frac{\p f}{\p z}(z)w_1+\frac{\p f}{\p \overline{z}}(z)\overline{w_1}\right)+\im\left(\frac{\p f}{\p z}(z)w_2+\frac{\p f}{\p \overline{z}}(z)\overline{w_2}\right).$$
On the other hand
$$2\left|\frac{\p f}{\p \overline{z}}(z)\right|=\left|\frac{\p f}{\p z}(z)(w_2-iw_1)+\frac{\p f}{\p \overline{z}}(z)\overline{w_2+iw_1}\right|.$$
Hence,
$$
\re\left(\frac{\p f}{\p z}(z)e^{i\g}\right)>\left|\frac{\p f}{\p \overline{z}}(z)\right|$$
is equivalent to
$$\re\left(\frac{\p f}{\p z}(z)w_1+\frac{\p f}{\p \overline{z}}(z)\overline{w_1}\right)+\im\left(\frac{\p f}{\p z}(z)w_2+\frac{\p f}{\p \overline{z}}(z)\overline{w_2}\right)>$$
$$\left|\frac{\p f}{\p z}(z)(w_2-iw_1)+\frac{\p f}{\p \overline{z}}(z)\overline{w_2+iw_1}\right|.
$$
Since $\re w_1\im w_2-\re w_2\im w_1=\cos^2\g+\sin^2\g=1$, obviously $\re w_1\im w_2\neq \re w_2\im w_1.$ The conclusion follows from Theorem \ref{t41}.
\end{proof}

\begin{rmk}\label{r41} Note that for $\g=0$ in Corollary \ref{c41}, we obtain Mocanu's injectivity theorem, see \cite{M}.
\end{rmk}

\begin{cor}\label{c42} Let $D\subseteq\C$ be open and convex and let $f:D\To\C$ be a function of class $C^1.$ Assume that there exist $\g\in\R$ such that, for all $z\in D$ it holds:
$$
\re\left(\frac{\p f}{\p \overline{ z}}(z)e^{i\g}\right)>\left|\frac{\p f}{\p {z}}(z)\right|.$$
  Then $f$ is injective.
\end{cor}
\begin{proof} Take $w_1=e^{i\g}$ and $w_2=-ie^{i\g}.$ Then $\re w_1\im w_2-\re w_2\im w_1=-\cos^2\g-\sin^2\g=-1$, hence $\re w_1\im w_2\neq \re w_2\im w_1.$  The rest of the proof is analogous to the proof of Corollary \ref{c41}.
\end{proof}

Let $D$ be open and connected. Recall that a function $f:D\subseteq \C\To\C$ is called holomorphic on $D$ if $f$ is derivable at every point of $D.$
Note, that in the case when $f$ is holomorphic on $D$, one has
$$\frac{\p f}{\p \overline{z}}(z)=0, \mbox{  for all }z\in D,$$ and
$$df_{z}(w)=\frac{\p f}{\p z}(z)w=f'(z)w,\mbox{ for all }z\in D\mbox{  and }w\in \C.$$ A holomorphic function which is also injective is called univalent. The next result is an extension of the univalence result of Alexander-Noshiro-Warschawski and Wolff.

\begin{cor}\label{c43} Let $D\subseteq\C$ be open and convex and let $f:D\To\C$ be a holomorphic function. Assume that there exist $w_1,w_2\in\C$ such that $\re w_1\im w_2\neq \re w_2\im w_1$ and for all $z\in D$ the following condition holds:
\begin{equation*}\label{e6}
\re f'(z)w_1+\im f'(z)w_2>|f'(z)||w_2-iw_1|.
\end{equation*}  Then $f$ is univalent.
\end{cor}
\begin{proof}According to Remark \ref{r4} the condition
$$\re\left(\frac{\p f}{\p z}(z)w_1+\frac{\p f}{\p \overline{z}}(z)\overline{w_1}\right)+\im\left(\frac{\p f}{\p z}(z)w_2+\frac{\p f}{\p \overline{z}}(z)\overline{w_2}\right)>$$
$$\left|\frac{\p f}{\p z}(z)(w_2-iw_1)+\frac{\p f}{\p \overline{z}}(z)\overline{w_2+iw_1}\right|,$$
 is equivalent to
$$\re df_z(w_1)+\im df_z(w_2)>|df_z(w_2)-idf_z(w_1)|.$$
On the other hand the latter relation is exactly
$$\re f'(z)w_1+\im f'(z)w_2>|f'(z)||w_2-iw_1|.$$
The conclusion follows from Theorem \ref{t41}.
\end{proof}

From the previous result one can easily obtain Alexander-Noshiro-Warschawski and Wolff univalence theorem (see \cite{A} and \cite{No,Wa,Wo}).
\begin{cor}\label{c44} Let $D\in\C$ be open and convex and let $f:D\To\C$ be a holomorphic function. Assume that there exist $\g\in\R$ such that, for all $z\in D$ it holds:
$$ \re f'(z)e^{i\g}>0.$$
Then $f$ is univalent.
\end{cor}
\begin{proof} Indeed, let $w_1=e^{i\g},\,w_2=ie^{i\g}.$ Then
$$\re f'(z)w_1+\im f'(z)w_2=2 \re f'(z)e^{i\g}.$$
	Obviously $|f'(z)||w_2-iw_1|=0,$ hence $ \re f'(z)e^{i\g}>0$ is equivalent to
$$\re f'(z)w_1+\im f'(z)w_2>|f'(z)||w_2-iw_1|.$$
Note that  $\re w_1\im w_2-\re w_2\im w_1=\cos^2\g+\sin^2\g=1$, hence $\re w_1\im w_2\neq \re w_2\im w_1.$ The conclusion follows from Corollary \ref{c43}.
\end{proof}

\noindent{\bf Acknowledgement.} This work was supported by a grant of the Romanian Ministry of Education, CNCS-UEFISCDI, project number PN-II-RU-PD-2012-3-0166.


\begin{thebibliography}{99}
\bibitem{KaPiLa1} G. Kassay, C. Pintea, S. L\' aszl\' o, Monotone
	operators and closed countable sets. Optimization, {\bf 60}(8-9)  (2011), 1059-1069	
\bibitem{Mb} P.T. Mocanu, T. Bulboac\u a, G. S\u al\u agean, Teoria Geometric\u a a Func\c tiilor Univalente, Casa C\u ar\c tii de
\c Stiin\c t\u a, Cluj-Napoca, 2006.
\bibitem{No}   K. Noshiro, On the theory of schlicht functions, J. Fac. Sei. Hokkaido Univ. (1) 2 (1934-
1935), 129-155.	
\bibitem{Wa} S. Warschawski, On the higher derivatives at the boundary in conformai mappings, Trans.
Amer. Math. Soc. 38 (1935), 310-340.
\bibitem{Wo} J. Wolff, L'int\'egrale d'une function holomorphe et \`a partie r\'eelle positive dans un demiplan est univalente, C. R. Acad. Sei. Paris 198 (1934), 1209-1210
\bibitem{KaPiLa2} G. Kassay,  C. Pintea, S. L\' aszl\' o,  Monotone operators and first category sets. Positivity, { 16} (2012),565-577 		
\bibitem{Pi} C. Pintea, Global injectivity conditions for planar maps, Monatshefte f\" ur Mathematik, doi: 10.1007/s00605-012-0474-x
\bibitem{M} P. T. Mocanu, Sufficient conditions of univalence for complex functions in the class
	C$^1$, Rev. Anal. Num\'er. Th\'eor Approx., 10 (1981), 75-79.

\bibitem{A} W. Alexander, Functions which map the interior of the unit circle upon simple regions,
	Ann. of Math., 17 (1915-1916), 12-22.

		\end{thebibliography}
	\end{document}